\documentclass[12pt]{amsart}
\usepackage{amssymb, latexsym}

\usepackage{amsthm}
\usepackage{amsmath,amscd}
\usepackage{amsfonts}
\usepackage[dvips]{graphicx}
\usepackage{amsthm}
\usepackage{amsmath,amscd}
\usepackage{amsfonts}
\usepackage[dvips]{graphicx}
\usepackage{color}
\numberwithin{equation}{section} 
\newcommand{\bean}{\begin{eqnarray*}}
\newcommand{\eean}{\end{eqnarray*}}
\newcommand{\be}{\begin{equation}}

\newcommand{\ee}{\end{equation}}
\newcommand{\bd}{\begin{displaymath}}
\newcommand{\ed}{\end{displaymath}}

\newcommand{\nab}{\vec{\nabla}_{x}}

\newcommand{\beq}{\begin{equation}}
\newcommand{\eeq}{\end{equation}}
\newcommand{\bea}{\begin{eqnarray}}
\newcommand{\eea}{\end{eqnarray}}

\newtheorem{lemma}{Lemma}[section]
\newtheorem{theorem}[lemma]{Theorem}
\newtheorem{definition}[lemma]{Definition}

\newcommand{\abs}[1]{\left\vert{#1}\right\vert}

\newcommand{\R}{\mathbb{R}}

\newcommand{\e}{\varepsilon}
\newcommand{\norm}[1]{\left\Vert#1\right\Vert}

\newtheorem{proposition}[lemma]{Proposition}

\def\({\left(}
\def\){\right)}
\def\1{\mathds{1}}
\def\a{{\alpha}}
\def\bo{\partial \Omega}

\def\dist{\text{dist}\ }
\def\div{\mathrm{div} \ }
\def\ep{\varepsilon}

\def\hal{\frac{1}{2}}

\def\lep{{|\mathrm{log }\ \ep|}}
\def\lepp{{|\mathrm{log }\ \ep'|}}

\def\l|{\left|}

\def\mc{\mathbb{C}}

\def\mr{\mathbb{R}}
\def\nab{\nabla}

\def\r|{\right|}

\def\he{h_{ex}}

\begin{document}

\title{Large vorticity stable solutions to the  Ginzburg-Landau equations}

\author{Andres Contreras and Sylvia Serfaty}

\begin{abstract}
We construct local minimizers to the Ginzburg-Landau functional of superconductivity whose number of vortices $N$ is prescribed and blows up as the parameter $\ep$, inverse of the Ginzburg-Landau parameter $\kappa$, tends to $0$. We treat the case of $N$ as large as $\lep$, and a wide range of intensity of external magnetic field.  The vortices of our solutions arrange themselves with uniform density over a subregion of the domain bounded by a ``free boundary" determined via an obstacle problem, and  asymptotically tend  to minimize the ``Coulombian renormalized energy" $W$ introduced in \cite{Vorlatt}.
  The method, inspired by \cite{Branches}, consists in minimizing the energy over a suitable subset of the functional space, and in showing that the minimum is achieved in the interior of the subset. It also relies heavily on refined asymptotic estimates for the Ginzburg-Landau energy obtained in \cite{Vorlatt}.
\end{abstract}\maketitle

\section{Introduction}
Let $\Omega\subseteq\R^2$ be a bounded and smooth domain. We  are interested in finding  local minimizers of the 2D Ginzburg-Landau functional of superconductivity
\begin{equation}\label{Barcarola}
G_{\e}(u,A):=\frac{1}{2}\int_{\Omega}\abs{(\nabla-iA)u}^2+\abs{\nabla\times A-h_{ex}}^2+\frac{(1-\abs{u}^2)^2}{2\e^2}.
\end{equation}
Here, the complex-valued function $u$ is a pseudo wave function, the ``order parameter" in physics,  that indicates the local state of the superconductor. The material is in the superconducting state in the regions where $\abs{u}\approx 1,$  and in the normal state where $\abs{u}\approx 0.$ The zeros of the order parameter carrying a non-zero degree are called {\it  vortices}.
The parameter $h_{ex}>0$ is the intensity of the  applied magnetic field, $A$ is the induced magnetic potential, $h:=\nabla\times A$ is the induced magnetic field, and finally $\e$ is the inverse of the ``Ginzburg-Landau parameter."
 For more background, we refer to standard physics textbooks, such as \cite{tinkham,tilley}, or to \cite{Book} and references therein.

We are interested in the asymptotic regime $\e\to 0.$  This corresponds to the characteristic lengthscale of the vortices going to $0$.

The vortices of a critical point of \eqref{Barcarola} are conveniently studied via the vorticity measure
  $$\mu(u,A):=\nabla\times j(u,A)+\nabla\times A,$$ where $j(u,A):=(iu, \nab_A u)$ denotes the superconducting current associated to a configuration and $\nab_A=\nab -i A$. We also
 write $j_\e:=j(u_\e,A_\e),$ for $(u_\e,A_\e)$ a pair depending on $\e.$
In the limit $\e\to 0$, the vorticity measure is well  approximated by a sum of Dirac masses at the vortex centers, for a precise statement see the Jacobian estimate \eqref{JacEst} below.

In what follows we write, for quantities $A_\e$ and $B_\e$ depending on $\e,$ $A_\e\lesssim B_\e$ when there is a constant $C$ independent of $\e$ such that $A_\e\leq CB_\e.$ We also write
 $A_\e\ll B_\e $ when $\lim_{\ep \to 0} \frac{A_\e}{B_\e}=0$.

The asymptotic study of critical points of \eqref{Barcarola}, as $\e\to 0$ for external fields $h_{ex}\ll\frac{1}{\e^2},$ has been carried out in a series of papers by Sandier and Serfaty, starting from  \cite{LocalminI,LocalminII}; continuing in \cite{Below,Bounded,LimVort}; the monograph \cite{Book}, and more recently \cite{Vorlatt}.
It was shown in \cite{Below} that for $h_{ex}=o\left(\lep\right)$ there is a critical value $\lambda_{*}$ of $\lambda=\lim_{\e\to 0}\frac{h_{ex}}{\lep}$ below which minimizers have no vortices. It was further shown in \cite{Bounded} that if $h_{ex}-\lambda_{*}\lep\lesssim\log\lep$ then the number of vortices of a minimizer remains bounded independent of $\e.$ For $  \log \lep \ll \he- \lambda_* \lep \ll \lep$, minimizers have $N$ vortices with $1\ll N\ll \he$ as shown in \cite{Book}, Chap. 9, while
 for $\lambda>\lambda_{*}$ possibly $=+\infty$, minimizers exhibit  a number $N$ proportional to $h_{ex} $  of vortices with uniform density in a subregion determined by an obstacle problem (see \cite{LimVort}, \cite{Book} Chap. 7, and \cite{Vorlatt}).

The behavior of critical points of \eqref{Barcarola} that are not necessarily minimizers of the energy, was the object of study of \cite{LimVort}. In \cite{LimVort} the vorticity measures $\mu(u_\e,A_\e)$ associated to critical points $(u_\e,A_\e)$ of \eqref{Barcarola}, normalized by the number of vortices,  are shown to converge weakly to a measure $\mu.$ The measure $\mu$ has the form $\mu:=-\Delta h+h,$  where $h$ is a ``stationary function" for a free boundary problem (see \eqref{Aux.1} below), and  satisfies in some weak sense the relation $\mu \nab h=0$.  This expresses that the average force acting on the vortices ($\nab h$) vanishes where the vortices are, i.e. on the support of $\mu$, thus allowing them to be at equilibrium.
This result was later proved under less restrictive assumptions in \cite{Book}, Chapter 13, and included cases  not previously covered.

 It is then natural to ask the converse, that is,  given a measure $\mu$ associated to a stationary function $h$ as above,  can one find $(u_\e,A_\e)$ critical points of \eqref{Barcarola} such that $\mu(u_\e,A_\e)$, suitably normalized, converges weakly to  $\mu?$
  We will answer this question partially, by building  solutions to the Ginzburg-Landau equations which are local (but not global) minimizers of the energy, and have a prescribed and divergent (as $\e\to 0$) number of vortices $N$. The same thing has already been accomplished for  a bounded number of vortices in \cite{Branches}, and also in \cite{Book}, Chap. 11 for $N$ slowly diverging as $\ep\to 0$.

   Note that the question of inverse problems, i.e. whether given a critical point
of a limiting energy, one can find  corresponding critical points of the original problem, has been addressed in a variety of contexts, and for Ginzburg-Landau in particular : see \cite{linlin,dulin} for the situation of  a finite number of vortices in 2D, and \cite{ABM,smz, jms} for lines of vortices in 3D. Our approach, which is not local inversion but rather local minimization, is more in the spirit of \cite{smz,jms} for vortex lines in 3D.
All previous results however, concerned only finite numbers of vortices.
Another related context where such a question is addressed is that of the Allen-Cahn equation, whose limiting energy is the perimeter functional, for results there see \cite{kdpw,ks,pr}.

The stable solutions found in \cite{Branches} and \cite{Book} Chap. 11, are shown to exist for a wide range of $\he$ (they are in fact globally minimizing for only one value of $\he$): they exist for fields of intensity up to $\sim\frac{1}{\e^{\alpha_0}},$ where $\alpha_0<\frac{1}{2},$ and also for fields almost as small as constant. They thus form  branches of solutions, each corresponding to a number $N$ of vortices, with $N$ bounded with $\e$, or even unbounded, but  subject to the restrictions $N^2\leq\eta h_{ex},$ $h_{ex}\leq N\e^{\frac{-2\eta}{N^2+\eta}},$ where $\eta>0$ is a small but fixed parameter.
These restrictions imply in particular (cf. \cite{Book}) that in any case $N\ll\left(\lep\right)^{\frac{1}{2}}$ and if $N\to\infty,$ then $h_{ex}\ll\frac{1}{\e^\alpha}$ for any $\alpha>0,$ thus leaving open the question of existence of stable solutions with $\left(\lep\right)^{\frac{1}{2}}\lesssim N\lesssim h_{ex}$ vortices.
We address this question and prove for the first time that branches of local minimizers with $N$ vortices continue to exist up to  $N =O(\he)$, with restrictions on  $N$ and $h_{ex}$ specified below,  see \eqref{kinderszenen}--\eqref{Asshexternal}.

The analysis in \cite{Branches,Book} consisted, roughly,  in minimizing the energy $G_\ep$ ``among configurations with $N$ vortices" and exploiting the quantization of $N$ through a careful expansion of the minimal energy of a configuration in terms of its vortices. This expansion came with an error $o(N^2).$ However, the method followed there only tolerates an error of $o\left(\lep\right)$ (hence at least the condition $N^2 \lesssim  \lep$)  which thus limited the range of applicability of the approach. What we do here is to follow the same general idea of local minimization but this time over a class which is defined so as to exploit the results and methods   of \cite{Vorlatt} to get a more precise energy expansion with an error of only  $o(N)$. In this way we can cover the range of $N$ even comparable to $\lep$, and fields as large as a power of $\frac{1}{\e}.$
\subsection{Statement of main result}

In order to state the main result of the article in more accurate terms, we first proceed to introducing some auxiliary function.

\subsubsection*{An Obstacle Problem}
Let $N$ be a positive number (typically an integer representing the number of vortices). Following  \cite{Vorlatt},  we denote by $h_{\e,N}$ the minimizer of
\begin{equation}\label{Aux.1}
\frac{1}{2}\int_{\Omega}\abs{\nabla h}^2+\abs{h-h_{ex}}^2,
\end{equation}
over the set of functions $h$ that are equal to $h_{ex}$ on the boundary and such that
\begin{equation}
\int_{\Omega}\abs{-\Delta h+h}=2\pi N.\nonumber
\end{equation}

As long as $2\pi N \le \he |\Omega|$ (an assumption we will need to make), this minimizer $h_{\e, N}$  exists and
it is the solution to an obstacle problem (see  Lemma 5.1 and Appendix A  in \cite{Vorlatt}).  There exists  a coincidence set $\omega_{\e,N}\subset \Omega$  and a number $0<m_{\e,N}\leq 1$ (for us the  normalized vortex density) such that
\beq\label{relaciones}
\mu_{\e,N}:=-\Delta h_{\e,N}+h_{\e,N}=h_{ex}\,m_{\e,N}\mathbf{1}_{\omega_{\e,N}}
\mbox{   and   }
h_{ex}\,m_{\e,N}\abs{\omega_{\e,N}}=2\pi N,
\eeq where $\mathbf{1}$ denotes the characteristic function of a set and $|\cdot |$ its area.
If $\Omega$ is convex, then $\omega_{\e,N}$ is also convex,  by a result in \cite{caff}. To avoid technical difficulties, we will assume in the sequel that $\Omega$ is convex.
It is also checked in \cite{Vorlatt} that $m_{\e,N}$ is bounded universally from below by a positive constant $m_0$ and that $m_{\e,N}\mapsto\abs{\omega_{\e,N}}$ is increasing.
Without loss of generality we can thus assume that
\beq\label{assumptiononN}
\lim_{\e\to 0}m_{\e,N}=:m\ge m_0>0,
\eeq
where $N$ may also depend on $\ep$.

As in \cite{Vorlatt} (cf. Lemma 5.2), the function $h_{\e,N}$ will serve to ``split'' the energy $G_{\e}$ in a convenient way.
The vorticity of the solutions we construct will ressemble $\mu_{\e, N}$, which is a constant over the subregion $\omega_{\e, N}$, hence we will have a uniform density of vortices there, and a vanishing one outside.







\subsubsection*{Range of fields and vorticity studied}

We consider fields and vortex numbers satisfying the following conditions
\beq\label{kinderszenen}
1\ll N\leq\min\left\{c\frac{h_{ex}\abs{\Omega}}{2\pi},K\log\frac{1}{\e\sqrt{h_{ex}}}\right\},
\eeq
for constants $0<c<1$  and $K>0$ independent of $\e$, and
\beq\label{Asshexternal}
  h_{ex}\lesssim\e^{-s_0}\mbox{ for some }s_0<\frac{1}{7}.
\eeq
Note that  $N\le K \log \frac{1}{\e\sqrt{h_{ex}} }$ could equivalently be written $N \le K \lep$ (changing $K$ if necessary).
Let us also point out that \eqref{kinderszenen} implies $h_{ex}\gg 1$
and   $0<m_0\leq m<1$ (see \eqref{relaciones}), and thus $1-m_{\ep, N}$ is always bounded away from $0$ for $\e$ small enough, and $h_{ex}(1-m_{\ep, N}) \gg 1$. We will make use of these facts repeatedly during the proofs.

Our main result is:
 \begin{theorem}\label{theotheo}
Let $\Omega\subseteq\mathbb{R}^2$ be a convex smooth bounded set.
Assume that $N=N(\e)$  a family of positive integers and $\he=\he(\ep)$  satisfy conditions   \eqref{kinderszenen}--\eqref{Asshexternal}.
Then, for $\e$ small enough (depending on the constants above), there exists a local minimizer $(u_\e,A_\e)$ of \eqref{Barcarola} such that the following holds:
\begin{enumerate}
\item[(i)]  There exists a measure $\nu_\ep$ of the form  $2\pi \sum_i d_i\delta_{a_i}$, with $d_i \in \mathbb{Z}$ and $a_i \in \Omega$, such that
\begin{equation} \label{JacEst}
\norm{\mu(u_\e, A_\e) -\nu_\e}_{(C_0^{0,1}(\Omega))^*}\le \ep^{\a}
\end{equation} for some  $\a>0$, and   we have
\beq\label{Azteca}
\nu_\e(\Omega)=2\pi N,
\eeq
i.e. the total vorticity of $u_\e$ is $N.$
\item[(ii)] The asymptotics
\beq\label{Maya}
G_{\e}(u_\e,A_\e)= G_{\e}^N+\pi N\log \frac{1}{\ep\sqrt{\he}}+NW_{m}-2\pi Nh_{ex}(1-m_{\e,N})+o(N),
\eeq
hold as $\e\to 0,$ where $W_m$ and $G_\e^N$ are  constants explicited in \eqref{WW} and \eqref{DefGepsn} respectively, and $m$ is as in \eqref{assumptiononN}.
\item[(iii)] For any $1<p<2,$
\beq\label{small}
\norm{\mu(u_\e,A_\e)-\mu_{\e,N}}_{W^{-1,p}(\Omega)}\le C_p \sqrt{N},
\eeq where $C_p$ depends only on $p$.
\item[(iv)] For $1<p<2$, denoting by $P_\e$ the probability measure on $L_{loc}^{p}(\R^2,\R^2)$ which is the push-forward of the normalized uniform measure on $\omega_{\e,N}$ by the map $x\mapsto \sqrt{\he}  j_\e(\sqrt{\he} (x+\cdot) ),$ \footnote{Here $j_\e$ is extended by $0$ outside $\Omega$.} we have that, up to subsequence, $P_\e$ converges weakly to a translation invariant probability measure $P$ on $L_{loc}^{p}(\R^2,\R^2)$ such that $P$-a.e. $j\in\mathcal{A}_{m}$ and $P$-a.e. $j$ minimizes $W$ over $\mathcal{A}_{m},$ where $W$ and $\mathcal{A}_m$ are defined just below.
\end{enumerate}
\end{theorem}

As announced,
this theorem proves the existence of locally minimizing solutions of Ginzburg-Landau with $N$ vortices. Moreover, their vortices tend to follow the average distribution $\mu_{\ep, N}$, and their superconducting currents, after blow-up at scale $\sqrt{\he}$, tend to minimize the ``Coulombian renormalized energy" $W$, introduced in \cite{Vorlatt}, whose definition we recall below.



\subsection{The renormalized energy $W$}
For each $m>0,$ consider the class of currents $j$ such that $\mbox{div} \ j=0$ and $\nabla\times j=\sum_{p\in\mathcal{P}}2\pi\delta_p-m$ for a discrete set $\mathcal{P},$ and $\#(\mathcal{P} \cap B_R)/R^2 $ is bounded for $R>1$.  We call this class $\mathcal{A}_m.$

For any family of sets $\{\mathbf{U}_R\}_{R>0} $ in $\R^2,$  $\chi_{\mathbf{U}_R}$ will be a family of associated positive cut-off functions satisfying
\beq
\abs{\nabla\chi_{\mathbf{U}_R}}\leq C,\,\,\,\,\mbox{Supp}(\chi_{\mathbf{U}_R})\subset\mathbf{U}_R,\,\,\,\,\chi_{\mathbf{U}_R}(x)=1\mbox{ if } d(x,(\mathbf{U}_{R})^{c})\geq1,\nonumber
\eeq
where the constant $C$ is independent of $R.$ We  specialize this to the family of balls $B_R$ of radius $R$ and center the origin.

\begin{definition} $W$ is defined, for $j\in\mathcal{A}_m,$ by
\beq
W(j)=\limsup_{R\to\infty}\frac{W(j,\chi_{B_R})}{\abs{B_R}},\nonumber
\eeq
where for any positive function $\chi$, we denote
\beq
W(j,\chi)=\lim_{\eta\to 0}\left(\frac{1}{2}\int_{\R^2\setminus\cup_{p\in\mathcal{P}}B(p,\eta)}\chi\abs{j}^2+\pi\log\eta\sum_{p\in\mathcal{P}}\chi(p)\right).\nonumber
\eeq
\end{definition}

For properties of $W$, we refer to \cite{Vorlatt}. It is proven there in particular that $\min_{j\in\mathcal{A}_{m}} W$ is achieved and finite.
We also set $\gamma$
\beq
\gamma:=\lim_{R\to\infty}\left[\frac{1}{2}\int rf'^2+\frac{(1-f^2)^2r}{2}+\frac{f^2}{r}\,dr-\pi\ln R\right],
\mbox{ constant introduced in \cite{BBH}},\nonumber
\eeq
where $f$ is such that $f(r)e^{i\theta}$ is the unique (cf. \cite{miro}), up to a phase shift, degree one radial solution of
\beq
-\Delta u+(1-\abs{u}^2)u=0\mbox{ in }\R^2.
\nonumber
\eeq 
We finally let  \beq\label{WW}
W_{m}:=\frac{2\pi}{m}\min_{j\in\mathcal{A}_{m}} W+\gamma.
\eeq

In \cite{Vorlatt}, it is conjectured that the minimum of $W$ over $\mathcal{A}_{m}$ is achieved for currents corresponding to a configuration of points in a perfect hexagonal lattice (of suitable volume). A small hint to this is provided there by the fact that $W$ is uniquely minimized,  among  perfect lattices of fixed volume, by the hexagonal lattice. In superconductivity, the hexagonal lattice also corresponds to the famous ``Abrikosov lattice", predicted from \eqref{Barcarola} and  observed in experiments. It was proven in \cite{Vorlatt} that global  minimizers of the Ginzburg-Landau functional \eqref{Barcarola} have currents which minimize $W$, after a suitable blow inside the free-boundary region that contains the vorticity.  If the conjecture about the minimum of $W$ is true, then one expects the vortices of global minimizers to arrange themselves in such an Abrikosov hexagonal lattice. From item (iv) of Theorem \ref{theotheo} the exact same thing is expected of the lo
 cally minimizing  solutions  found here. The only qualitative difference with the global minimizers is in the total number of vortices $N$, which manifests itself in a different $\mu_{\e, N}$ hence a different  local density of vortices and a different location for the free-boundary that encloses them.


\subsection{The local minimization method}
The local minimization first relies on the same  {\it energy splitting} introduced in \cite{Vorlatt}, which we now recall. For any $A$, set $$
A_1:=A-\nabla^{\perp}h_{\e,N}$$ where $h_{\e, N}$ is the solution of the obstacle problem \eqref{Aux.1} above.
Let us  also define accordingly
\begin{equation}
J_{\e,N}(u,A):=\frac{1}{2}\int_{\Omega}\abs{(\nabla-iA)u}^2+\abs{\nabla\times A-\mu_{\e,N}}^2+\frac{(1-\abs{u}^2)^2}{2\e^2}\,dx,\nonumber
\end{equation} with $\mu_{\ep,N}$ as in \eqref{relaciones},
the configuration independent quantity
\beq\label{DefGepsn}
G_\e^N:=\norm{h_{\e,N}-h_{ex}}^2_{H^1(\Omega)}+2\pi N(m_{\e,N}-1)h_{ex},
\eeq
and finally the function $\xi:=h_{ex}-h_{\e,N}$ (implicitly depending on $\ep$ and $N$, note that it is nonnegative in $\Omega$  and  vanishes on $\partial \Omega$).  We denote  \begin{equation}\label{ximax}
\xi_{\max}=\max_{\Omega} \xi = h_{ex}(1-m_{\e,N}).\end{equation}

\begin{lemma}
For any $(u,A)$, $G_{\e}(u,A)$ can be decomposed as:
\beq\label{Decomp.Split}
G_\e(u,A)=G_\e^N+J_{\e,N}(u,A_1)-\int_{\Omega}\xi\,\mu(u,A_1)+
\frac{1}{2}\int_{\Omega}(\abs{u}^2-1)\abs{\nabla h_{\e,N}}^2.
\eeq
\end{lemma}

\begin{proof}
This splitting formula was written with  slightly different notation in \cite{Vorlatt}.
Decomposition \eqref{Decomp.Split} follows easily from the relations:
\beq
\abs{(\nabla-iA)u}^2=\abs{(\nabla-iA_1)u}^2+\abs{u}^2\abs{\nabla^{\perp}h_{\e,N}}^2-2j(u,A_1)\cdot\nabla^{\perp} h_{\e,N},\nonumber
\eeq
and
\begin{eqnarray}
\abs{\nabla\times A-h_{ex}}^2
&=&
\abs{\nabla\times A_1+\Delta h_{\e,N}-h_{\e,N}+h_{\e,N}-h_{ex}}^2 \nonumber\\
&=& \abs{\nabla\times A_1-\mu_{\e,N}}^2+\abs{h_{\e,N}-h_{ex}}^2+2(\nabla\times A_1-\mu_{\e,N})(h_{\e,N}-h_{ex}).\nonumber
\end{eqnarray}
\end{proof}

The last term in the right-hand side of \eqref{Decomp.Split} will be proven later to be $o(1)$ for the configurations of interest to us.

As explained above for the case of smaller values of $N$, we still
 seek our local minimizers as  configurations that minimize $G_\ep$  among configurations with $N$ vortices. The question is to define (following \cite{Branches}) a suitable  set
corresponding to this heuristic idea.
This is achieved here by minimizing  $G_\ep$ over

\begin{multline}\label{defS}
S_{\ep, N}^\beta := \bigg\{(u,A)\in H^1(\Omega, \mc)\times H^1(\Omega, \mr^2)
\mbox{ s.t. } J_{\e,N}(u,A_1)\leq\pi N\log \frac{1}{\ep\sqrt{\he}}+NW_m+\beta N;\\
\int_{\Omega}\xi\mu\geq 2\pi Nh_{ex}(1-m_{\e,N})-\beta N\bigg \},
\end{multline}
where $\beta$ is chosen  to be a small enough positive constant. The idea of this class  is that it allows for at most $N$ vortices when $\beta $ is small,  it leaves enough room for a particular test function to be admissible (the precision with  which we can compute its energy is $o(N)$ while the error allowed is $\gtrsim$ N), and  it can  easily be related to the Ginzburg-Landau energy $G_\ep$ via \eqref{Decomp.Split}.
As in \cite{Branches,Book}, the task is then, exploiting the quantization of $N$, to show that $\min_{S_{\e, N}^\beta} G_\ep$ is achieved in the interior of $S_{\e, N}^\beta$, thus providing a local minimizer of $G_\ep$, and to show it has the other desired properties. 
The class $S_{\e,N}^\beta$ differs from those in \cite{Branches, Book}. Indeed, following these works it would be tempting to try minimizing over the set of configurations $(u,A)$ such that:
\beq
\abs{ J_{\e,N}(u,A_1)-\left(\pi N\log \frac{1}{\ep\sqrt{\he}}+NW_m\right)}\leq\beta N.\nonumber
\eeq
However, if $ J_{\e,N}(u,A_1)=\pi N\log \frac{1}{\ep\sqrt{\he}}+NW_m-\beta N$  we cannot conclude that  $u$ has strictly less than $N$ vortices, which is crucial to assure that such a configuration on the boundary of the set would have an energy that exceeds that of a local minimizer.
If one replaces $\beta N$ by $\beta\log \frac{1}{\ep\sqrt{\he}}$ a similar problem occurs; in this scenario one is able to prove that if $(u,A)$ is in the corresponding lower boundary then it has at most $N-1$ vortices, but the  $-\beta\log \frac{1}{\ep\sqrt{\he}}$  in the value of $J(u,A_1)$ makes the energy of the hypothetical configuration possibly  drop below the one in Proposition \ref{lemmaparacota}, rendering that approach doomed.
To overcome this difficulty (and this is a novelty of the paper) we proceed by bounding the total vorticity from below indirectly by requiring the elements of $S_{\e,N}^\beta$ to satisfy
\beq
\int_\Omega\xi\mu(u,A_1)\geq 2\pi h_{ex}(1-m_{\e,N})N-\beta N,\nonumber
\eeq
this implies there are at least $N$ vortices as long as there are no vortices with negative degree outside the free boundary $\omega_{\e,N}.$
This is proven to be the case in Proposition \ref{cuadripropo}.

Finally, the question of treating the case of $N$ larger than the order  $\lep$ remains open, we believe it would require much more precise estimates on the energy than those provided by \cite{Vorlatt} that we use here, and this seems out of reach for the moment.
\\

The plan of the paper is as follows.
 In Section 2 we introduce notation and recall the Jacobian estimate.
 In Section 3 we  construct an explicit configuration with $N$ vortices uniformly distributed according to $\mu_{\ep, N}$ and a superconducting current  almost minimizing $W$,   we check the test configuration  belongs to the admissible class, and by computing its energy $G_\ep$, we obtain an almost optimal upper bound.
In Section 4 we analyze general properties of the elements of $S_{\e,N}^\beta$, in particular characterize their total vorticity as well as give lower bounds for the different components of the splitting of $G_\e.$ We conclude with the proof of Theorem \ref{theotheo} in Section 5, by showing minimizers do not lie on the boundary of the admissible class.
\\

{\it Acknowledgment:}  Both authors are supported by S.S's EURYI award.


\section{Preliminaries}

We will mostly work in blown-up coordinates at the characteristic lengthscale $1/\sqrt{\he}$.
Let
\beq\label{prelimeq}
\e':=\e\sqrt{h_{ex}}\mbox{ and }x':=\sqrt{h_{ex}}\,x.
\eeq
The domains $\Omega_{\e}$ and $(\omega_{\e,N})'$ correspond to $\Omega$ and $\omega_{\e,N}$ in the blown-up coordinates $x'.$ We transform the rest of the variables in a natural fashion, that is

\beq\label{foranymag}
u'(x')=u(x), \,\,
\xi'(x')=\xi(x) \mbox{ and }
A'(x')=\frac{1}{\sqrt{h_{ex}}}A(x),
\eeq
for any magnetic potential $A$.

The blown up currents and measures are defined in terms of these quantities in a clear way, also denoted with a prime.

We define, for $U\subset\Omega_\e,$ the sets:
\beq \label{22bis}
\widehat{U}:=\{x\in\Omega_\e/d(x,U)\leq 1\}
\eeq
\beq \label{22ter}
\check{U}:=\{x \in \Omega_\e/ d(x, \partial U) \ge 1\}
\eeq
and
\beq\label{tildeU}
\widetilde{U}:=\{x\in\Omega_\e/d(x,\partial U)\geq 2\}.
\eeq

We now recall the basic concentration estimate relating the Jacobian and the total vorticity, \`a la Jerrard-Soner \cite{js}.

\begin{theorem} (cf. \cite{Vorlatt}, Theorem 5)
Assume $G_\e(u_\e,A_\e)\lesssim (\e')^{-s},$ where $s<1$.  Then there exists a measure $\nu_\e'$ that depends only on $u_\e$ (not on $A_\e$) of the form  $\nu_\e=2\pi\sum_{i} d_i\,\delta_{a_i},$ for $d_i \in \mathbb{Z}$ and $a_i \in \Omega_\e$, with
\beq\label{Dubliners}
\norm{\mu'-\nu_\e'}_{(C_{0}^{0,1}(\Omega_\e))^{*}}\lesssim(\e\sqrt{h_{ex}})^{\frac{1}{2}}G_\e(u_\e,A_\e),
\eeq
and there exists a constant $C_1>0$ such that for any measurable set $E\subseteq\Omega_\e$
\beq\label{measu}
\abs{\nu_\e'}(E)\leq C_1\frac{\int_{\hat{E}}e_\e}{\abs{\log\e'}},
\eeq
where
\beq \label{ee}
e_\e(u,A):=\abs{\nabla_{A} u}^2+\abs{\nabla\times A}^2+\frac{(1-\abs{u}^2)^2}{2(\e')^2}.
\eeq
\end{theorem}


\section{Local minimization class and energy upper bound}

As mentioned, we follow the idea of  \cite{Branches}, and Chapter 11 of \cite{Book}, and seek for local minimizers of $G_\e$ among configurations with ``$N$
vortices'' by minimizing $G_\ep$ over  the class $S_{\e, N}^\beta$, cf. \eqref{defS}, with $\beta$ to be determined later.
From now on we write $({u_\e},{A_\e})$ an element in the Argmin of  $G_\e$ in $S_{\e,N}^{\beta}.$
The fact that such a $({u_\e},{A_\e})$  actually exists follows from classical arguments (a proof can be adapted from  \cite{Branches} or Chapter 11 in \cite{Book}).
The task is then to show a minimizer cannot lie on the boundary.

We first show that the classes $S_{\e,N}^\beta$ are non-empty and provide a priori estimates on the energy of a minimizer in those classes.

\begin{proposition}\label{lemmaparacota}
The classes $S_{\e,N}^\beta$ are non-empty.
Furthermore, if $(u_\e,A_\e)$ is a minimizer of $G_\e$ over $S_{\e,N}^\beta,$ then, as $\e \to 0$, 
\beq\label{Ubound}
G_{\e}({u_\e},{A_\e})\leq G_{\e}^N+\pi N\lepp+NW_{m}-2\pi Nh_{ex}(1-m_{\e,N})+o(N).
\eeq
\end{proposition}

\begin{proof}

The proof relies on computing the energy of a carefully built test configuration $(\widehat{u_\e},\widehat{A_\e}).$
The construction is carried out in Theorem 6 in \cite{Vorlatt} where the argument makes use of the fact that in their case $N$ corresponds to the number of vortices of a global minimizer.
We show here that that the construction still works for $N$ satisfying \eqref{kinderszenen} and also make explicit why the test configuration belongs to $S_{\e,N}^\beta$, an issue not present in the global minimization setting of \cite{Vorlatt}.

In Theorem 6 of \cite{Vorlatt}, a test configuration $(\widehat{u_\e},\widehat{A_\e})$ is built in the following way.
\vskip.2in
\underline{Step 1:}
\vskip.1in
A discrete set of points $\mathcal{P}_\e,$ $\mathcal{P}_\e\subseteq\omega_{\e,N}$ and $\abs{\mathcal{P}_\e}=N$ is chosen.
A current $\widehat{\jmath_\e}$ is built that satisfies
\beq\label{ooooo}
\limsup_{\eta\to 0}\frac{1}{m_{\e,N}h_{ex}\abs{\omega_{\e,N}}}\left(\frac{1}{2}\int_{\Omega\setminus \bigcup_{p\in\mathcal{P}_\e} B\left(p,\eta(m_{\e,N}h_{ex})^{-\frac{1}{2}}\right)}\abs{\widehat{\jmath_\e}}^2+\pi N\log\eta\right)
\leq\min_{\mathcal{A}_1}W+\mathcal{O}_\e(1).
\eeq

The current $\widehat{\jmath_\e}$ is constructed under the assumption that  $\Omega$ is convex (Proposition 4.2 and Corollary 7.4 in \cite{Vorlatt}). In addition it is necessary that the coincidence set $\omega_{\e,N}$ be well contained inside $\Omega.$ More precisely one must guarantee
\beq\label{dcstb}
\dist(\omega_{\e,N},\Omega^{c})\gtrsim(m_{\e,N}h_{ex})^{-\frac{1}{2}}
\eeq

In \cite{Vorlatt} an explicit lower bound for $\dist(\omega_{\e,N},\Omega^{c})$ was proven in the case where $N$ is the number of vortices corresponding to global minimizers of $G_\e.$
In our case we     recall that \eqref{kinderszenen} implies $m<1$ and therefore $\sqrt{1-m_{\e,N}}$ is bounded from below by a positive number. From Appendix A in \cite{Vorlatt}, we know that
\beq\label{123}
\|\nabla h_{\e,N}\|_{\infty}\lesssim h_{ex}\sqrt{1-m_{\e,N}}.
\eeq
Then since $h_{\e,N}$ is equal to $m_{\e,N}\,h_{ex}$ on $\omega_{\e,N}$ and to $h_{ex}$ on the boundary, $\xi = h_{ex}-  h_{\e, N}$ and \eqref{ximax},  one has:
\beq
\dist(\omega_{\e,N},\Omega^{c})\gtrsim\frac{\xi_{\max}}{\norm{\nabla h_{\e,N}}_{L^\infty}}\gtrsim\sqrt{1-m_{\e,N}}.\nonumber
\eeq
On the other hand, $m_{\e,N}\geq m_0>0$, $1-m_{\ep, N}$ is bounded below, and $h_{ex} \gg 1$, therefore
\beq
\sqrt{1-m_{\e,N}}\gtrsim\frac{1}{\sqrt{m_{\e,N}\,h_{ex}}},\nonumber
\eeq
which proves \eqref{dcstb} in our case.
\vskip.2in
\underline{Step 2:}
\vskip.1in
The functions $\widehat{u_\e}$ and $\widehat{A_\e}$ are then built satisfying
\beq\label{o1}
\abs{\widehat{u_\e}}=1\mbox{ on }\Omega\setminus\cup_{p\in\mathcal{P}_\e}B(p,M\e)
\mbox{ and }\cup_{p\in\mathcal{P}_\e}B(p,M\e)\subset\omega_{\e,N}.
\eeq

\beq\label{o2}
\widehat{A_\e}:=\widehat{A_{1,\e}}+\nabla^{\perp}h_{\e,N},
\eeq
where $\widehat{A_{1,\e}}$ is chosen to satisfy $\nabla \times  \widehat{A_{1,\e}} = \mu_{\e, N}$  and
\beq\label{o3}
\frac{1}{2}\int_{B(p,M \e)}\abs{\nabla_{\widehat{A_{1,\e}}}\widehat{u_\e}}^2+\frac{(1-\abs{\widehat{u_\e}}^2)^2}{2\e^2}
=\pi\log M+\gamma+o_M(1)+o_\e(1),
\eeq
for all $p\in\mathcal{P}_\e.$

\beq\label{o4}
\frac{1}{2}\int_{B\left(p,\eta(m_{\e,N}\,h_{ex})^{-\frac{1}{2}}\right)\setminus B(p,M \e)}\abs{\nabla_{\widehat{A_{1,\e}}}\widehat{u_\e}}^2+\frac{(1-\abs{\widehat{u_\e}}^2)^2}{2\e^2}
\leq\pi\log\left(\frac{\eta\sqrt{m_{\e,N}\,h_{ex}}}{M\e}\right)+\mathcal{O}(\eta).
\eeq

\beq\label{o5}
\int_\Omega\mu(\widehat{u_\e},\widehat{A_{1,\e}})=\int_\Omega\mu_{\e,N}=2\pi N,\mbox{ and  }\mu(\widehat{u_\e},\widehat{A_{1,\e}})\equiv 0\mbox{ on }(\omega_{\e,N})^c.
\eeq
From this it easily follows that $(\widehat{u_\e},\widehat{A_\e})\in S_{\e,N}^\beta.$
Indeed

\begin{eqnarray}
J_{\e,N}(\widehat{u_\e},\widehat{A_{1,\e}})
&=&
\frac{1}{2}\int_\Omega\abs{\nabla_{\widehat{A_{1,\e}}}\widehat{u_\e}}^2+\abs{\nabla\times\widehat{A_{1,\e}}-\mu_{\e,N}}^2+\frac{(1-\abs{\widehat{u_\e}}^2)^2}{2\e^2}\nonumber\\
&=&\frac{1}{2}\int_\Omega\abs{\nabla_{\widehat{A_{1,\e}}}\widehat{u_\e}}^2+\frac{(1-\abs{\widehat{u_\e}}^2)^2}{2\e^2}.\nonumber
\end{eqnarray}
Above, the second term disappears because of \eqref{o2}.

The identity $\abs{\nabla_A u}^2=\abs{\nabla\abs{u}}^2+\abs{u}^2\abs{j}^2$ implies, together with \eqref{o1} that
\beq
\abs{\nabla_{\widehat{A_{1,\e}}}\widehat{u_\e}}^2=\abs{\widehat{\jmath_\e}}^2\mbox{ on }\Omega\setminus\cup_{p\in\mathcal{P}_\e}B(p,M_\e).\nonumber
\eeq
Therefore, combining \eqref{ooooo}, \eqref{o3} and \eqref{o4}, one gets after taking $M\to\infty,$ next $\eta\to 0,$
\beq\label{belong1}
J_{\e,N}(\widehat{u_\e},\widehat{A_{1,\e}})
\leq
\pi N\abs{\log\e}+2\pi N\min_{\mathcal{A}_1}W+\pi N\log\frac{1}{\sqrt{m_{\e,N}\,h_{ex}}}+N\gamma+o(N).
\eeq
By assumption $m_{\e,N}\to m,$ therefore $N\abs{\log m_{\e,N}}=N\abs{\log m}+o(N).$
According to (1.17) and (1.18) of \cite{Vorlatt}, by scaling, one has
\beq
\min_{\mathcal{A}_1}W=\frac{1}{m}\min_{\mathcal{A}_m}W+\frac{1}{4}\log m,\nonumber
\eeq
Thus, by definition of $W_m$ and $\e',$ \eqref{belong1} is equal to
\beq\label{belong2}
J_{\e,N}(\widehat{u_\e},\widehat{A_{1,\e}})\leq\pi N\abs{\log\e'}+NW_m+o(N).
\eeq
This shows that  $(\widehat{u_\e},\widehat{A_{1,\e}})$ satisfies the upper bound in the definition of $S_{\e,N}^\beta.$ To see it also satisfies the lower bound, we see by \eqref{o5} that

\begin{eqnarray}\label{belong3}
&&\int_\Omega\xi\mu(\widehat{u_\e},\widehat{A_{1,\e}})=\int_{\omega_{\e,N}}\xi\mu(\widehat{u_\e},\widehat{A_{1,\e}})
=\xi_{\max}\int_{\omega_{\e,N}}\mu(\widehat{u_\e},\widehat{A_{1,\e}})\nonumber\\
&&=\xi_{\max}\int_{\Omega}\mu(\widehat{u_\e},\widehat{A_{1,\e}})
=\xi_{\max}\int_{\Omega}\mu_{\e,N}=2\pi Nh_{ex}(1-m_{\e,N}).
\end{eqnarray}
By \eqref{belong2} and \eqref{belong3} we obtain that  $(\widehat{u_\e},\widehat{A_\e})\in S_{\e,N}^\beta.$

To conclude, from the splitting formula \eqref{Decomp.Split}, we see that
\beq
G_{\e}(\widehat{u_{\e}},\widehat{A_{\e}})=G_{\e}^N+J_{\e,N}(\widehat{u_\e},\widehat{A_{1,\e}})-\int_\Omega\xi\mu(\widehat{u_\e},\widehat{A_{1,\e}})+\frac{1}{2}\int_{\Omega}(\abs{\widehat{u_\e}}^2-1)\abs{\nabla h_{\e,N}}^2.\nonumber
\eeq
The last term is identically zero because $\abs{\nabla h_{\e,N}}=0$ on $\omega_{\e,N},$ while $\abs{\widehat{u_\e}}=1$ on $(\omega_{\e,N})^c$ by \eqref{o1}.
Hence, by \eqref{belong2} and \eqref{belong3} again, we see that
\beq
G_{\e}(\hat{u_{\e}},\hat{A_{\e}})\leq G_{\e}^N+\pi N\lepp + N W_m-2\pi Nh_{ex}(1-m_{\e,N})+o(N),
\nonumber
\eeq
and the proposition is proved.
\end{proof}

In the next section, the lower bound component of Theorem \ref{theotheo} requires a very precise approximation of the vorticity measure by $\nu_\e.$ We also need a better control on the rightmost term in the splitting formula \eqref{Decomp.Split}.

\begin{lemma}\label{lemmaB}

Let $(u_\e,A_\e)$ be a minimizing pair on $S_{\e,N}^\beta.$
Assume $N$ and $h_{ex}$ satisfy \eqref{kinderszenen} and \eqref{Asshexternal}. Then there exists an $\alpha>0$ such that
\beq\label{lemmaeq}
\norm{\mu'-\nu_\e'}_{(C_{0}^{0,1}(\Omega_\e))^{*}}=o(\e^\alpha)
\eeq
and
\beq\label{ytambien}
\int_\Omega(1-\abs{u_\e}^2)\abs{\nabla h_{\e,N}}^2=o(1).
\eeq
\end{lemma}

\begin{proof}
It is easy to see that $G_\e^N\lesssim h_{ex}^2,$ while the rest of the positive terms in \eqref{Ubound} are $o(h_{ex}^2).$ Therefore
\beq
G_\e(u_\e,A_\e)\lesssim h_{ex}^2.\nonumber
\eeq
From \eqref{Dubliners} and \eqref{Asshexternal} we can deduce

\beq\label{deduce}
\norm{\mu'-\nu_\e'}_{(C_{0}^{0,1}(\Omega_\e))^{*}}\lesssim\e^{\frac{1}{2}}h_{ex}^{\frac{1}{4}}\cdot h_{ex}^2\lesssim\e^{\frac{1}{2}-\frac{9}{4}s_0},
\eeq
The fact that this is $o(\e^\alpha)$ for some $\alpha>0$ follows from \eqref{Asshexternal}.
Thus, we have proved \eqref{lemmaeq}, and \eqref{JacEst} also follows.
For the second assertion, Cauchy-Schwarz yields
\begin{eqnarray}
\abs{\int_\Omega(\abs{u_\e}^2-1)\abs{\nabla h_{\e,N}}^2}\leq
\left(\int_\Omega(1-\abs{u_\e}^2)^2\right)^{\frac{1}{2}}\|\nabla h_{\e,N}\|^2_{\infty}\abs{\Omega}^{\frac{1}{2}}
\lesssim
\sqrt{\e^2G_\e(u_\e,A_\e)}\|\nabla h_{\e,N}\|^2_{\infty}\nonumber
\end{eqnarray}

Using \eqref{123} and the bound $h_{ex}^2\gtrsim G_\e(u_\e,A_\e),$ the above leads to
\beq
\int(1-\abs{u_\e}^2)\abs{\nabla h_{\e,N}}^2\lesssim \e h_{ex}^3\leq\e^{(1-3s_0)}=o(1),\nonumber
\eeq
by \eqref{Asshexternal}.

\end{proof}


\section{Further Properties of minimizers and lower bound for $G_\e(u_\e,A_\e)$}
 The next proposition, which is the crucial part of the proof,  shows that the elements of $S_{\e,N}^\beta$ have $N$ vortices and also that the upper bound in \eqref{Ubound} is sharp. 
 
 At this point, it is convenient to introduce the notion of ``vorticity adapted'' family of functions. To that end, for a function $\psi_\e,$ let
$\omega^{\psi_\e}:=\left\{x\in\Omega_{\e}\, s.t. \,\psi_\e(x)=\frac{1}{2}\lepp\right\}.$
We will restrict to considering families $\{\psi_\e\}_{\e>0}$ such that
\beq\label{u}
\|\nabla\psi_\e\|_{\infty}=o(\lepp).
\eeq
Now, define the class:
\begin{eqnarray}
\mathcal{VA}&:=&
\bigg\{\{\psi_\ep\}_{\e>0}\mbox{ such that }\psi_\e:\Omega_\e\to\mathbb{R}, \psi_\e\mbox{ satisfies \eqref{u}, }\omega^{\psi_\e}\supset(\omega_{\e,N})';
\nonumber\\
&&0 \le \psi_\e(x)<\frac{1}{2}\lepp  \mbox{ if }x\in\Omega_{\e}\setminus\omega^{\psi_\e};
\psi_\e=0\mbox{ on }\partial\Omega_\e\bigg\}\nonumber
\end{eqnarray}
Abusing notation, we sometimes call $\psi_\e$ a vorticity adapted function when $\{\psi_\e\}_{\e>0}$ is a vorticity adapted family.
This notion will help us in making use of a family of estimates available in \cite{Vorlatt} in a transparent way.
The proof in \cite{Vorlatt} of such estimates, is presented for a particular vorticity adapted function, while here we use the estimates in their full generality. Ultimately, these estimates, together with the quantization of the vorticity measure, will prove that $S_{\e,N}^\beta$ consists only of functions with $N$ essential vortices.

In what follows, a $+$ or $-$ superscript denotes the positive resp. negative part of a function or measure.

\begin{proposition}\label{cuadripropo}
Let $N$ and $h_{ex}$ be sequences satisfying \eqref{kinderszenen} and \eqref{Asshexternal}. Assume $(u^\e,A^\e)\in S_{\e,N}^\beta$. 
Then, for $\e $ small enough, we have:
\begin{enumerate}
\item[(i)] There are no vortices with negative associated degree outside the free boundary, i.e.:
\beq
\nu_\e=\nu_\e^{+}\mbox{ on }(\omega_{\e,N})^c.\nonumber
\eeq

\item[(ii)] The total vorticity is $N,$ in other words:
\beq
\nu_\e(\Omega)=2\pi N.\nonumber
\eeq

\item[(iii)]
\beq
\int_\Omega\xi\mu(u^\e,A_1^\e)\leq 2\pi Nh_{ex}(1-m_{\e,N})+o(1).\nonumber
\eeq

\item[(iv)]
\beq
J_{\e,N}(u^\e,A_1^\e)\geq \pi N\abs{\log\e'}+NW_m+o(N).\nonumber
\eeq
\end{enumerate}
\end{proposition}
\vskip.2in
{\it Proof.}
Recall that $(u^\e,A^\e)\in S_{\e,N}^\beta$ means that
\beq\label{SchW}
J_{\e,N}(u^\e,A_1^\e)\leq\pi N\abs{\log\e'}+NW_m+\beta N,
\eeq
and
\beq\label{SchD}
\int_\Omega\xi\mu(u^\e,A_1^\e)\geq 2\pi Nh_{ex}(1-m_{\e,N})-\beta N.
\eeq

The proof relies on lower bounds on quantities of the form
\beq \label{form}
J_{\e,N}(u^\e,A_1^\e)-\int_\Omega\psi_\e(\sqrt{h_{ex}}\,\cdot)\mu(u^\e,A_1^\e),\nonumber
\eeq
where $\{\psi_\e\}_{\e>0}\in\mathcal{VA}$.

\underline{Step 1:} First we claim we may assume that
 $(u^\e,A^\e)$ satisfies
\beq\label{OCD}
\mbox{div}\,j(u^\e,A^\e)=0.
\eeq
This assumption is justified by the following: we claim there is a $A_0^\e$ such that $\mbox{div}\, j(u^\e,A_0^\e)=0$ and
\begin{eqnarray*}
J_{\e,N}(u^\e,A_1^\e)-\int_{\Omega}\psi_\e(\sqrt{h_{ex}}\,\cdot)\,\mu(u^\e,A_1^\e)&\geq&
J_{\e,N}(u^\e,A_{0,1}^\e)-\int_{\Omega}\psi_\e(\sqrt{h_{ex}}\,\cdot)\,\mu(u^\e,A_{0,1}^\e)+o(1),\nonumber\\
\end{eqnarray*}
holds for any  $\{\psi_\e\}_{\e>0}\in\mathcal{VA}.$
Indeed, letting $A_0^\e$ be such that $G_\e(u^\e,A_0^\e)=\min_{A\in H^1(\Omega, \mathbb{R}^2)}G_\e(u^\e,A)$ (such a minimizer is easily seen to exist), by minimality  $(u^\e,A_0^\e)$ satisfies the second Ginzburg-Landau equations, that is $-\nabla^{\perp}(\nabla\times A_0^\e)=(iu^\e,\nabla_{A_0^\e}u^\e),$ whence $j(u^\e,A_0^\e)$ is divergence free.

By definition of $A_0^\e,$ $G_\e(u^\e,A_0^\e)\leq G_\e(u^\e,A^\e)\lesssim h_{ex}^2.$
As in the proof of  Lemma \ref{lemmaB} this upper bound implies
\beq
\int_\Omega(1- \abs{u^\e}^2)\abs{\nabla h_{\e,N}}^2=o(1).\nonumber
\eeq
We can use the splitting formula \eqref{Decomp.Split} to conclude:
\begin{eqnarray*}
G_\e^N+J_{\e,N}(u^\e,A_1^\e)-\int_\Omega\xi\mu(u^\e,A_1^\e)&\geq&
G_\e^N+J_{\e,N}(u^\e,A_{0,1}^\e)-\int_\Omega\xi\mu(u^\e,A_{0,1}^\e)+o(1).\nonumber\\
\end{eqnarray*}
From this, we can assert that

\begin{multline*}
J_{\e,N}(u^\e,A_1^\e)-\int_\Omega\psi(\sqrt{h_{ex}}\,\cdot)\mu(u^\e,A_1^\e)\geq
J_{\e,N}(u^\e,A_{0,1}^\e)-\int_\Omega\psi_\e(\sqrt{h_{ex}}\,\cdot)
\mu(u^\e,A_{0,1}^\e)\nonumber\\
-\int_\Omega\left(\psi_\e(\sqrt{h_{ex}}\,\cdot)-\xi\right)\left(\mu(u^\e,A_1^\e)
-\mu(u^\e,A_{0,1}^\e)\right)+o(1).
\end{multline*}
If the field $h_{ex}$ is not too large, the Jacobian estimate \eqref{Dubliners} implies that the last term is $o(1).$
To see this, recall that the measure $\nu_\e$ in \eqref{Dubliners} does not depend on $A_\e$ and therefore:
\begin{eqnarray*}
&&\abs{\int_\Omega\left(\psi(\sqrt{h_{ex}}\,\cdot)-\xi\right)\left(\mu(u^\e,A_{0,1}^\e)-\mu(u^\e,A_{1}^\e)\right)}\nonumber\\
&&\quad  \leq
\abs{\int_\Omega\left(\psi(\sqrt{h_{ex}}\,\cdot)-\xi\right)\left(\mu(u^\e,A_{0,1}^\e)-
d\nu_\e\right)}
+\abs{\int_\Omega\left(\psi(\sqrt{h_{ex}}\,\cdot)-\xi\right)\left(d\nu_\e-\mu(u^\e,A_{1}^\e)\right)}\nonumber\\
&&\quad
\lesssim\norm{\psi(\sqrt{h_{ex}}\,\cdot)-\xi}_{\infty}\cdot
\max\bigg\{\norm{\nu_\e-\mu(u^\e,A_{0,1}^\e)}_{\left(C_0^{0,1}(\Omega)\right)^{*}},\norm{\nu_\e-\mu(u^\e,A_{1}^\e)}_{\left(C_0^{0,1}(\Omega)\right)^{*}}\bigg\}\nonumber\\
&& \quad
\lesssim (\norm{\xi}_\infty+\|\nabla\psi\|_\infty) \cdot\left(\sqrt{\e'}G_\e(u^\e,A^\e)\right)
\lesssim (h_{ex}+h_{ex}^{\frac{1}{2}} \lep) \,\e^{\frac{1}{2}}\,h_{ex}^{\frac{1}{4}}\,h_{ex}^2 =o(1),
\end{eqnarray*}
thanks to  \eqref{Asshexternal}.  This proves the claim.
\vskip .3cm

\underline{Step 2:} {\bf  Lower bound with  vorticity adapted functions.}
We switch to blown-up coordinates in order to appeal to the proof of Proposition 6.1 in \cite{Vorlatt}.
Recalling that the Jacobian estimate allows us to approximate $\mu'$ by $\nu_\e',$ we can replace bounding below \eqref{form} by bounding below 
\begin{eqnarray*}J_{\e,N}'((u^{\e})',(A_{1}^{\e})')-\int_{\Omega_\e}\psi_\e\,d\nu_\e'.
\end{eqnarray*}

But it was proved in \cite{Vorlatt} that
there exists a function $g_\e'$ depending only on $((u^{\e})',(A^{\e})')$ such that
for any vorticity adapted function $\psi_\e$, and for some constant $C_1>0$,

\begin{eqnarray}\label{frio}
J_{\e,N}'((u^{\e})',(A_{1}^{\e})')-\int_{\Omega_\e}\psi_\e\,d\nu_\e'
&\geq& \frac{1}{6}\int_{\Omega_\e}\left(1-\frac{2\psi_\e}{\lepp}\right)e_\e+\frac{1}{C_1}(g_{\e'})^{+}\left((\widetilde{\omega_{\e,N}}')^c\right)\nonumber\\
&&+g_{\e'}\left((\widetilde{\omega_{\e,N}})'\right)+o(\abs{(\omega_{\e,N})'}),
\end{eqnarray} where $e_\ep$ is as in \eqref{ee}.

The argument can be found in (6.28)--(6.35) of \cite{Vorlatt} where $\zeta'_\e$ can be replaced by $\psi_\e;$ indeed it can be readily checked that the only fact about $\zeta'_\e$ that was used in the proof was that it was vorticity adapted. Note that here we also use the fact that $\div  \, j(u^\e, A^\e)=0$, obtained in Step 1.

We also record the followings facts from \cite{Vorlatt, improv}
\begin{enumerate}
\item[1.] There exists a family of disjoint balls $\{B_i\}_{i\in I}$ of bounded total radius such that the support of $\nu_\e'$ is contained in $\cup_{i\in I}B_i,$ a universal constant $C_2$ and functions $\{g_\e^{B_i}\}$ such that if $\nu_\e'(B_i)<0$
\beq\label{lownegdeg}
\int_{B_i}g_{\e'}^{B_i}\geq\frac{1}{8}\left(\lepp-C_2\right)\abs{\nu_\e(B_i)}\geq\frac{1}{16}\lepp\abs{\nu_\e(B_i)}
\eeq
(The family $\{B_i\}_{i\in I}$ is built in \cite{improv}, where \eqref{lownegdeg} is proved in Proposition 3.2)
\vskip.2in

\item[2.] There exists a universal constant $C_3$ such that
\beq\label{lownegdeg1}
(g_{\e'})^{+}(x)\geq\frac{3}{4}\left(\sum_{\{i, \, \nu_\e'(B_i)<0\}}\left(g_{\e'}^{B_i}(x)-C_3\right)\mathbf{1}_{B_i}\right)
\eeq
(This is a consequence of the definition of $g_{\e'}$ in \cite{Vorlatt}, combined with (3.22) in \cite{improv} and \eqref{lownegdeg}).
\vskip.3in

\item[3.]   Let $P$ be the probability measure in item $(iv)$ of Theorem \ref{theotheo}.
Because $(u^\e,A^\e)$ satisfies the bounds in the definition of $S_{\e,N}^\beta,$ and since the field $h_{ex}$ is not too large (cf. \eqref{Asshexternal}), the conclusions of Proposition 6.3 in \cite{Vorlatt} hold in our case with the density $m$ instead of $m_\lambda$. 
Assumption \eqref{OCD} then  allows us to assert (recalling the definitions \eqref{WW} and \eqref{tildeU})
\beq\label{calor}
g_\e'(\widetilde{\omega_{\e,N}}')
\geq N\left(\frac{2\pi}{m}\int W(j)\,dP(j)+\gamma+o(1)\right)+o(\abs{(\omega_{\e,N})'})
\geq NW_m+o(N),
\eeq since, in view of \eqref{relaciones}
\begin{eqnarray}\label{sizeo}
\abs{(\omega_{\e,N})'}=\abs{\omega_{\e,N}}h_{ex}=\frac{2\pi N}{m_{\e,N}}
\leq\frac{2\pi N}{m_0}=\mathcal{O}(N).
\end{eqnarray}
\end{enumerate}
Using 
\eqref{frio} together with \eqref{calor}--\eqref{sizeo}, we have thus found that for any $\psi_\e \in \mathcal{VA}$,
\beq\label{eq.Ren.energy}
J_{\e,N}'((u^\e)',(A_{1}^{\e})')-\int_{\Omega_\e}\psi_\e\,d\nu_\e'
\geq \frac{1}{6}\int_{\Omega_\e}\left(1-\frac{2\psi_\e}{\lepp}\right)e_\e+NW_m+o(N).
\eeq

\vskip.3in
\underline{Step 3:} {\bf All vortices of $u^\e$ are contained in a neighborhood of $(\omega_{\e,N})'.$}
\vskip.2in
We will quite often need to use the fact that for any measurable set $A,$
\beq\label{quantnu}
\nu_\e'(A)\in 2\pi\mathbb{Z}, \mbox{ while }\abs{(\nu_\e')^{+}(A)},\,\abs{(\nu_\e')^{-}(A)},\,\abs{\nu_\e'}(A)\in 2\pi\mathbb{N}.
\eeq

  In what remains, $\alpha$ is set to be the constant
\beq 
\alpha:=\frac{\abs{\log\e'}}{2h_{ex}(1-m_{\e,N})}.\nonumber
\eeq Note that $\alpha$ depends  on $\ep $ and $N$ but we omit to indicate this dependence.
Also, from here on, all the primed quantities denote the blown up versions of the unprimed ones, according to the transformations \eqref{prelimeq} -- \eqref{foranymag}.
We begin by noticing that  $\alpha\xi'\in\mathcal{VA}.$ Indeed, since $h_{ex}\gg 1,$ and $m=\lim_{\e\to 0}m_{\e,N}<1$ one can readily see, using \eqref{123} and the fact that $\he (1-m_{\e, N}) \gg 1$, that
\begin{multline}\label{Arpe}
\norm{\nabla(\alpha\xi')}_{\infty}
=\mathcal{O}\left(\frac{\lepp}{h_{ex}(1-m_{\e,N})}\right)\cdot
\norm{\nabla\xi}_{\infty}\cdot\frac{1}{\sqrt{h_{ex}}}
=\mathcal{O}\left(\frac{\lepp}{(h_{ex})^{\frac{3}{2}}(1-m_{\e,N})}\right)
\cdot\norm{\nabla h_{\e,N}}_{\infty}
\\
\lesssim\mathcal{O}\left(\frac{\lepp}{(h_{ex})^{\frac{3}{2}}(1-m_{\e,N})}
\right)h_{ex}\sqrt{1-m_{\e,N}}
\lesssim\mathcal{O}\left(\frac{\lepp}{\sqrt{h_{ex}(1-m_{\e,N})}}\right)=
o(\abs{\log\e'}).
\end{multline}

By \eqref{SchW}, \eqref{SchD} and definition of $\alpha,$ we find that
\begin{multline}\label{Hades}
J_{\e,N}'((u^{\e})',(A_{1}^{\e})')-\alpha\int_{\Omega_\e}
\xi'\mu'((u^{\e})',(A_{1}^{\e})')\\
\leq
\pi N\lepp+NW_m+\beta N-\alpha( 2\pi N h_{ex}(1-m_{\e,N})-\beta N)\\
=(1+\alpha)\beta N+NW_m.
\end{multline}

 Conversely, the bound  \eqref{eq.Ren.energy} applied to $\psi_\e=\alpha \xi'$ 
together with \eqref{Hades}  will allow us to rule out the presence of vortices with negative degrees where $\xi'$ is small. In turn, we shall see that this implies an upper bound on the total vorticity $\nu_\e(\Omega_\e)$.

 As announced, we  apply \eqref{eq.Ren.energy} to $\psi_\e=\alpha \xi'$  which is indeed vorticity adapted, combining with the Jacobian estimate \eqref{Dubliners}, \eqref{Hades} and \eqref{eq.Ren.energy} we obtain that
 \begin{eqnarray}
\nonumber (1+\alpha)\beta N+NW_m
& \geq&
J_{\e,N}'((u^\e)',(A^\e_1)')-\alpha\int_{\Omega_\e}\xi'\mu'((u^\e)',(A_1^\e)')\\ \nonumber
&\geq &
J_{\e,N}'((u^\e)',(A^\e_1)')-\alpha\int_{\Omega_\e}\xi'\,d\nu_\e' +\mathcal{O}\left(\|\nabla(\alpha\xi')\|_{\infty}
\norm{\mu'-\nu_\e'}_{\left(C_0^{0,1}(\Omega_\e)\right)^{*}}\right)\\
\label{ultratumba}
&\geq &  \frac{1}{6}\int_{\Omega_\e}\left(1-\frac{2\alpha\xi'}{\lepp}\right)e_\e+NW_m+o(N)
+\mathcal{O}\left(\abs{\log\e'}\e^{\frac{1}{2}}\,h_{ex}^{\frac{1}{4}}\,G_\e(u^\e,A^\e)\right).
\end{eqnarray} 
The last term is $o(1)$ because $G_\e(u^\e,A^\e)\lesssim h_{ex}^2$ and $h_{ex}$ satisfies \eqref{Asshexternal}. Thus, we obtain
\beq\label{from}
\frac{1}{6}\int_{\Omega_{\e}}\left(1-\frac{2\alpha\xi'}{\lepp}\right)e_\e\leq\beta(1+\alpha)N+o(N).
\eeq
For any $Z>0$ small, we introduce the following neighborhood of $\bo$:
\beq
U^Z:=\left\{x'/\xi'(x')\leq\left(1-\frac{6C_1\left( K+\frac{c\abs{\Omega}(1-m_{\e,N})}{2\pi}\right)Z}{\pi}\right)h_{ex}(1-m_{\e,N})\right\},\nonumber
\eeq
where $C_1$ and $K$ are the constants in \eqref{measu} and \eqref{kinderszenen} respectively.
From the definition of $\alpha,$ it follows that
\beq
1-\frac{2\alpha\xi'}{\abs{\log\e'}}
\geq\frac{6C_1\left( K+\frac{c\abs{\Omega}(1-m_{\e,N})}{2\pi}\right)\beta}{\pi}\;
\mathbf{1}_{U^\beta}.\nonumber
\eeq
Using this together with \eqref{measu}, we obtain (cf. \eqref{22ter}) 
\begin{eqnarray}\nonumber
\int_{\Omega_\e}\left(1-\frac{2\alpha\xi'}{\abs{\log\e'}}\right)e_\e
&\geq&\frac{6C_1\left( K+\frac{c\abs{\Omega}(1-m_{\e,N})}{2\pi}\right)\beta}{\pi}\;
\int_{U^\beta}e_\e\\\label{0.18response}
&\geq&\frac{6\left( K+\frac{c\abs{\Omega}(1-m_{\e,N})}{2\pi}\right)\beta}{\pi}\;
\abs{\nu_\e'}(\check{U^\beta})\abs{\log\e'}.
\end{eqnarray}

On the other hand, by \eqref{kinderszenen}, $N\leq K\abs{\log\e'}$ and $\frac{\alpha N}{\lepp}  \leq\frac{c\abs{\Omega}}{2\pi(1-m_{\e, N})},$ so combining \eqref{from} and \eqref{0.18response}, we get $\frac{\abs{\nu_\e'}(\check{U^\beta})}{\pi}\leq 1+o\left( \frac{N}{\lepp}\right) <2$ (for $\e $ small enough),
which implies
\beq\label{vortnearfb}
\abs{\nu_\e'}(\check{U^\beta})=0,
\eeq
in light of \eqref{quantnu}.

\vskip.3in
\underline{Step 3:} {\bf There are no vortices with negative  degree outside $(\omega_{\e,N})'$}
\vskip.2in
We turn to the proof of item (i).
We invoke \eqref{frio} applied to $\alpha \xi'$ and \eqref{calor}--\eqref{sizeo} again to get
\begin{equation*}
J_{\e,N}'((u^\e)',(A_{1}^{\e})')-\int_{\Omega_\e}\alpha \xi' \,d\nu_\e'
\geq \frac{1}{C_1}(g_{\e'})^{+}((\widetilde{\omega_{\e,N}}')^c)+NW_m+o(N),
\end{equation*}
which, arguing as in   \eqref{ultratumba} yields
\beq\label{chiarina}
(1+\alpha)\beta N+o(N)\geq \frac{1}{C_1}(g_{\e'})^{+}((\widetilde{\omega_{\e,N}}')^c).
\eeq
Recall that from \eqref{kinderszenen} we have $\alpha\beta N\lesssim\beta\abs{\log\e'}\left(\frac{N}{h_{ex}}\right)\lesssim\beta\abs{\log\e'}$ hence $(1+\alpha) \beta N +o(N)  \lesssim \beta \lepp +o(\lepp).$

On the other hand, 
from \eqref{lownegdeg} and \eqref{lownegdeg1} one can bound $(g_{\e'})^{+}((\widetilde{\omega_{\e,N}}')^c)$ from below by (letting $C$  denote a generic positive constant)
\begin{equation}
(g_{\e'})^{+}((\widetilde{\omega_{\e,N}}')^c)
\ge
C\lepp\abs{\nu_\e^{-}((\widetilde{\omega_{\e,N}}')^c)}
-C_3\frac{\abs{\nu_\e^{-}((\widetilde{\omega_{\e,N}}')^c)}}{2\pi}\sum_{i\in I}\abs{B_i}
\gtrsim
\lepp\abs{\nu_\e^{-}((\widetilde{\omega_{\e,N}}')^c)}.\nonumber
\end{equation}

Going back to \eqref{chiarina}, this implies $\beta+o(1) \geq C\abs{\nu_\e^{-}\left((\widetilde{\omega_{\e,N}}')^c\right)},$ where $C$ is a positive constant independent of $\beta$ and $\e.$ Since $\abs{\nu_\e^{-}}\in 2\pi\mathbb{N},$  we may choose $\beta$ small so that for any $\ep $ small enough, it imposes \beq\label{nonegoutfb}
\abs{(\nu_\e')^{-}\left((\widetilde{\omega_{\e,N}}')^c\right)}=0.
\eeq
Since $ (\widetilde{\omega_{\e,N}}')^c$ contains $ (\omega_{\e,N}')^c$, we have thus  proved item (i).

\vskip.3in
\underline{Step 5:} {\bf The total vorticity of $u^\e$ is $N$}
\vskip.2in
We proceed, once more, applying \eqref{eq.Ren.energy} only this time to a vorticity adapted function $\psi_\e$ such that $\omega^{\psi_\e}\supseteq\Omega_\e\setminus U^\beta$. Such a function exists.
Indeed, by \eqref{Arpe}, for $x \in \partial (\Omega_\e \backslash U^\beta)$ we have $|\alpha \xi'(x)|\le o(\lepp) d(x, \partial \Omega_\e)$ while $\alpha \xi'(x)=\hal (1-\mathcal{O}(1)\beta) \lepp$. It follows that $d(\partial (\Omega_\e \backslash U^\beta), \partial \Omega_\e) \gg 1$ as $\e \to 0$ (if $\beta $ is a small enough constant), and thus $d(\partial (\widehat{\Omega_\e\backslash U^\beta}), \partial \Omega_\e) \gg 1$. One can therefore build a family of functions $\psi_\e$ which are equal to $\hal \lepp$ in $\widehat{\Omega_\e\backslash U^\beta}=\Omega_\e \backslash \check{U^\beta}$, which vanish on $\partial \Omega_\e$ and satisfy $\|\nabla \psi_\e\|_\infty =o(\lepp)$, i.e. belongs to $\mathcal{VA}$.

Applying \eqref{eq.Ren.energy} to that family yields
\beq
J_{\e,N}'((u^\e)',(A_{1}^{\e})')-\int_{\Omega_\e}\psi_\e\,d\nu_\e'\geq NW_m+o(N).\nonumber
\eeq
We can then appeal to \eqref{lemmaeq} again to obtain
\begin{multline*}
J_{\e,N}'((u^\e)',(A_{1}^{\e})')-\int_{\Omega_\e}\psi_\e\mu'((u^\e)',(A_{1}^{\e})')
\\ \geq
J_{\e,N}'((u^\e)',(A_{1}^{\e})')-\int_{\Omega_\e}\psi_\e^\beta\,d\nu_\e'+\mathcal{O}
(\norm{\nabla\psi_\e}_{\infty}\,\e^\alpha)\geq NW_m+o(N).
\end{multline*}
Since $(u^\e,A^\e)\in S_{\e,N}^\beta$, we must have
\begin{eqnarray}\label{FUN}
\pi N\abs{\log\e'}+\beta N\geq
J_{\e,N}'((u^\e)',(A_{1}^{\e})')-NW_m\geq\int_{\Omega_\e}\psi_\e\,d\nu_\e'+o(N).
\end{eqnarray}
Now, by \eqref{vortnearfb} and definition of $\psi_\e$
\beq\label{superduper}
\int_{\Omega_\e}\psi_\e\,d\nu_\e'
=\int_{\Omega_\e\setminus \check{U^\beta}}\psi_\e\,d\nu_\e'
=\int_{\widehat{\Omega_\e\setminus U^\beta}}\frac{1}{2}\abs{\log\e'}\,d\nu_\e'
=\frac{1}{2}\abs{\log\e'}\,\nu_\e'(\Omega_\e\setminus \check{ U^\beta}).
\eeq
Plugging this into \eqref{FUN} leads us to
\beq
N+\frac{\beta}{\pi}\frac{N}{\lepp}\geq \frac{1}{2\pi}\nu_\e'(\Omega_\e\setminus \check{ U^\beta})
+o\left(\frac{N}{\lepp}\right).\nonumber
\eeq
This implies that for $\e$ small enough,
\beq\label{Thisyields}
N\geq
\frac{\nu_\e'(\Omega_\e\setminus \check{U^\beta})}{2\pi}=\frac{\nu_\e'(\Omega_\e)}{2\pi},
\eeq
because of \eqref{quantnu} and \eqref{vortnearfb}.
We now show  that there is actually equality in \eqref{Thisyields}.
To that end we first prove
\beq
\int_{\Omega_\e}\xi'\mu((u^\e)',(A_1^\e)')\leq\nu_\e(\Omega_\e)
\xi'_{\max}+o(1).\nonumber
\eeq
This is not immediately clear since $\xi$ is not constant and $\nu_\e$ is not necessarily equal to $\nu_\e^{+}.$
However, from \eqref{Asshexternal} and \eqref{JacEst}, we have
\begin{equation*}
\int\xi'\mu((u^\e)',(A_1^\e)')=\int_{\Omega_\e}\xi'\,d\nu_\e'+\mathcal{O}(\e^{\frac{1}{2}}\,h_{ex}^{\frac{1}{4}+2}\abs{\nabla\xi'}_{\infty})
= \int_{\widetilde{\omega_{\e,N}}'}\xi'\,d\nu_\e'
+ \int_{(\widetilde{\omega_{\e,N}}')^c}\xi'\,d\nu_\e'
+o(\e^{\alpha}).
\end{equation*}
On the other hand, 
from \eqref{nonegoutfb},

\begin{equation*}
 \int_{(\widetilde{\omega_{\e,N}}')^c}\xi'\,d\nu_\e'
= \int_{(\widetilde{\omega_{\e,N}}')^c}\xi'\,d(\nu_\e')^+  
\leq \xi_{\max}(\nu_\e')^{+}((\widetilde{\omega_{\e,N}}')^c).
\end{equation*}

Also,
$$
 \int_{(\widetilde{\omega_{\e,N}}')}\xi'\,d\nu_\e'
= \xi_{\max}\nu_\e'(\widetilde{\omega_{\e,N}}')
$$
because $\xi'= \xi_{\max}\mbox{ on }(\widetilde{\omega_{\e,N}}').$
Collecting these estimates we are led to 
\beq\label{song}
\int_{\Omega_\e}\xi'\mu'((u^\e)',(A_1^\e)')
\leq\xi_{\max}\left(\nu_\e'(\widetilde{\omega_{\e,N}}')+(\nu_\e')^{+}((\widetilde{\omega_{\e,N}}')^c)\right)+o(1)
=\xi_{\max}\nu_\e'(\Omega_\e)+o(1)
\eeq
where the last equality follows from \eqref{nonegoutfb}.
Consequently $\nu_\e'(\Omega_\e)>2\pi(N-1),$ for otherwise we would have (cf. \eqref{ximax})
\beq
\int_{\Omega_\e}\xi'\mu'((u^\e)',(A_1^\e)')\leq   2\pi (N-1) \xi_{\max}+o(1) = 2\pi(N-1)h_{ex}(1-m_{\e,N})+o(1)\nonumber
\eeq
 and the right-hand side is strictly smaller than $2\pi Nh_{ex}(1-m_{\e,N})-\beta N,$ if $\beta$ is small enough, by virtue of \eqref{kinderszenen}. This contradicts $(u^\e,A_1^\e)\in S_{\e,N}^\beta.$ Therefore, $\nu_\e'(\Omega_\e)\geq 2\pi N$ (recall $\nu_\e'(\Omega_\e)\in 2\pi\mathbb{Z}.$) Thus,  with \eqref{Thisyields}, we have proved item (ii).
\vskip.3in
\underline{Step 6:} 
{\bf Proof of (iii) and (iv)}
\vskip.2in

Item (iii) follows readily from \eqref{song} since we  already know that $\nu_\e'(\Omega_\e)=2\pi N.$
For item (iv), note that from \eqref{vortnearfb}, \eqref{FUN} and \eqref{superduper} we deduce
\begin{eqnarray}
&&J_{\e,N}'((u^\e)',(A_1^\e)')
\geq
\int_{\Omega_\e}\psi_\e\,d\nu_\e'+NW_m+o(N)\nonumber\\
&&=\frac{1}{2}\abs{\log\e'}\nu_\e'(\Omega_\e\setminus \check{U^\beta})+NW_m+o(N)
= \pi N\abs{\log\e'}+NW_m+o(N).\nonumber
\end{eqnarray}
The proposition is proved.
\qed

\section{Conclusion}

In this section we finish the proof of our main result.
\vskip.2in

\underline{Proof of Theorem \ref{theotheo}}
\vskip.1in

As mentioned earlier, a minimizer in $S_{\e,N}^\beta$ exists, so to show it actually corresponds to a local minimizer, it only remains to prove $(u_\e,A_\e)\notin\partial S_{\e,N}^\beta.$

We decompose $\partial S_{\e,N}^\beta=(\partial S_{\e,N}^\beta)^{+}\bigcup(\partial S_{\e,N}^\beta)^{-}$  where
\beq
(\partial S_{\e,N}^\beta)^{+}:=\big\{(u,A)\in H^1\times H^1
\mbox{ such that: } J_{\e,N}(u,A_1)=\pi N\abs{\log\e'}+NW_m+\beta N \big \}\nonumber
\eeq
and
\beq
(\partial S_{\e,N}^\beta)^{-}:=\big\{(u,A)\in H^1\times H^1
\mbox{ such that: } \int_{\Omega}\xi\mu(u,A_1)= 2\pi Nh_{ex}(1-m_{\e,N})-\beta N \big \}.\nonumber
\eeq
Since $(u_\e,A_\e)$ is a minimizer in $S_{\e,N}^\beta,$ we may apply Lemma \ref{lemmaB} and  Proposition \ref{cuadripropo}.
By \eqref{Decomp.Split} and \eqref{ytambien}, we have
\beq\label{ForAsympExp}
G_\e(u_\e,A_\e)=G_\e^N+J_{\e,N}(u_\e,A_{1,\e})-\int_\Omega\xi\mu(u_\e,A_{1,\e})+o(1).
\eeq
If $(u_\e,A_\e)\in(\partial S_{\e,N}^\beta)^{+}$ then by (iii) in Proposition \ref{cuadripropo}, we have
\beq
G_\e(u_\e,A_\e)\geq G_\e^N+\pi N\abs{\log\e'}+NW_m+\beta N-2\pi Nh_{ex}(1-m_{\e,N})+o(N).\nonumber
\eeq
If instead,  $(u_\e,A_\e)\in(\partial S_{\e,N}^\beta)^{-},$ invoking (iv) in Proposition \ref{cuadripropo} yields
\beq
G_\e(u_\e,A_\e)\geq G_\e^N+\pi N\abs{\log\e'}+NW_m-\left(2\pi Nh_{ex}(1-m_{\e,N})-\beta N\right)+o(N).\nonumber
\eeq

In either case there is a contradiction with the upper bound \eqref{Ubound}, so necessarily $(u_\e,A_\e)\notin\partial S_{\e,N}^\beta,$ whence $(u_{\e}, A_{\e})$ is a local minimizer.

From \eqref{ForAsympExp}, (iii) and (iv) in Proposition \ref{cuadripropo}, we must have
\beq\label{LowforAsymp}
G_\e(u_\e,A_\e)\geq G_\e^N+\pi N\abs{\log\e'}+NW_m-2\pi Nh_{ex}(1-m_{\e,N})+o(N).
\eeq
Thus, the asymptotic expansion \eqref{Maya} follows from \eqref{LowforAsymp} and \eqref{Ubound}.
\eqref{JacEst}, \eqref{Azteca} were proved in Lemma \ref{lemmaB} and Proposition \ref{cuadripropo} respectively.

We prove item (iii). It essentially follows the proof of (1.34) in \cite{Vorlatt}. First, arguing as in item 2 of Proposition 6.1 in \cite{Vorlatt}, replacing $F_\e'$ with $J_{\e,N}'-\alpha\int_{\Omega_\e}\xi'\mu',$  we can sandwich the term  $\int_\Omega |\nabla \times A_{1, \e} -\mu
_{\e, N}|^2$ in $J_\e(u_\e, A_{1,\e})$  between the upper and lower bound of \eqref{ultratumba}  and hence control it
by $\mathcal{O}(N)$, i.e.
\beq \label{nmu}
\norm{\nabla\times A_{1,\e}-\mu_{\e,N}}_{L^2}
\lesssim\sqrt{N}
\eeq
Using the same item of Proposition 6.1 in \cite{Vorlatt} also allows to control, for $1<p<2$,
\beq 
\int_{\Omega_\e} |j(u_\e', A_{1,\e}')|^p \le C_p  \left(J_{\e,N}'( u_\e', A_{1,\e}' ) -\alpha\int_{\Omega_\e}\xi'\mu'+|\omega_{\e, N}'|\right)\le C (N+ |\omega_{\e, N}'|).\nonumber
\eeq
Rescaling we find that
for any $1<p<2$, 
\beq\label{contj2}
\|j(u_\e, A_{1,\e})\|_{L^p(\Omega)} \lesssim \he^{\hal-\frac{1}{p} } \left(N^{1/p}+ |\omega_{\e,N}'|^{1/p}\right) \lesssim  \he^{\hal-\frac{1}{p} } N^{1/p}
\lesssim N^{\frac{1}{2}}
\eeq
thanks to  \eqref{sizeo}, then \eqref{kinderszenen} giving $N\lesssim \he$.
On the other hand, by direct calculation we have 
$$\mu(u_\e,A_\e)-\mu(u_\e, A_{1,\e})= \nabla \times ((1-|u_\e|^2)\nabla^{\perp}
 h_{\ep, N}) $$
from \eqref{Asshexternal}and \eqref{123}
\beq
\norm{(1-\abs{u_\e}^2)\abs{\nabla h_{\e,N}}}_{L^2}
\lesssim\sqrt{\e^2\,G_\e}\cdot\abs{\nabla h_{\e,N}}_{\infty}
\lesssim\e\, h_{ex}^2=o(1),\nonumber
\eeq
from which one obtains
\beq\label{hastaaca}
\norm{\mu(u_\e,A_{1,\e})-\mu(u_\e,A_\e)}_{W^{-1,p}}=o(1).\eeq We may now write, using the triangle inequality
\begin{multline*}
\norm{\mu(u_\e,A_\e)-\mu_{\e,N}}_{W^{-1,p}}
\leq
\norm{\mu(u,A_\e)-\mu(u_\e,A_{1,\e})}_{W^{-1,p}}
+\norm{\mu_{\e,N}-\nabla\times A_{1,\e} }_{W^{-1,p}}\\
+ \norm{\mu(u_\e, A_{1,\e}) - \nabla \times A_{1, \e} }_{W^{-1,p}}
\lesssim N
\end{multline*}
where we have used  \eqref{nmu} (together with an embedding inequality), \eqref{hastaaca} and the
fact that $\mu(u_\e, A_{1,\e}) - \nabla \times A_{1,\e}=\nabla \times
 j(u_\e, A_{1,\e})$ combined with \eqref{contj2}.
Consequently, item (iii) is proved.

We turn to the proof of (iv).
Since $\alpha\xi'\in\mathcal{VA},$ \eqref{frio}, \eqref{calor} and \eqref{Maya} allow us to assert
\beq
NW_m+o(N)\geq
 N\left(\frac{2\pi}{m}\int W(j)\,dP(j)+\gamma+o(1)\right)\nonumber
\eeq
which in turns translates into
\beq
W_m\geq
\frac{2\pi}{m}\int W(j)\,dP(j)+\gamma.\nonumber
\eeq

To conclude, since $P$-a.e. $j\in\mathcal{A}_{m}$ (as given by \cite{Vorlatt}) this implies $P$-a.e. $j$ minimizes $W$ over $\mathcal{A}_{m}.$ The result follows.

\qed

\vskip 1cm

{\sc Andres Contreras}\\
UPMC  Univ  Paris 06, UMR 7598 Laboratoire Jacques-Louis Lions,\\
 Paris, F-75005 France ;\\
 CNRS, UMR 7598 LJLL, Paris, F-75005 France \\
 {\tt contreras@ann.jussieu.fr}
 \\ \\
 
{\sc Sylvia Serfaty}\\
UPMC Univ  Paris 06, UMR 7598 Laboratoire Jacques-Louis Lions,\\
 Paris, F-75005 France ;\\
 CNRS, UMR 7598 LJLL, Paris, F-75005 France \\
 \&  Courant Institute, New York University\\
251 Mercer st, NY NY 10012, USA\\
{\tt serfaty@ann.jussieu.fr}

\end{document}